\documentclass[12pt,reqno]{amsart}
\usepackage{amsmath,amssymb,amsthm,enumerate}
\usepackage{stmaryrd,tikz,cite}
\usepackage[colorlinks]{hyperref}\usepackage[a4paper,margin=25mm,top=30mm,bottom=35mm]{geometry}
\usepackage[english]{babel}
\usepackage[utf8]{inputenc}
\usepackage[T1]{fontenc}

\newtheorem{lemma}{Lemma}
\newtheorem{theorem}[lemma]{Theorem}
\newtheorem{proposition}[lemma]{Proposition}
\theoremstyle{definition}
\newtheorem{example}[lemma]{Example}

\newcommand{\ra}{\rightarrow}
\newcommand{\al}[1]{\mathbf{#1}}
\newcommand{\el}{$\ell$-}
\newcommand{\cd}{\cdot}
\newcommand{\bs}{\backslash}
\newcommand{\xal}[2]{\mathcal{#1}(\mathbf{#2})}
\newcommand{\xxal}[3]{\mathcal{#1}(\mathcal{#2}(\mathbf{#3}))}

\begin{document}

\title[A note on residuated po-groupoids]{A note on residuated po-groupoids and lattices with antitone involutions}

\author[I. Chajda \and J. K\"{u}hr]{Ivan Chajda* \and Jan K\"{u}hr**}

\newcommand{\acr}{\newline\indent}

\address{\llap{*\,}Department of Algebra and Geometry\acr
Faculty of Science\acr
Palack\'{y} University in Olomouc\acr
17. listopadu 12, CZ-77146 Olomouc\acr
Czech Republic}
\email{ivan.chajda@upol.cz}

\address{\llap{**\,}Department of Algebra and Geometry\acr
Faculty of Science\acr
Palack\'{y} University in Olomouc\acr
17. listopadu 12, CZ-77146 Olomouc\acr
Czech Republic}
\email{jan.kuhr@upol.cz}

\thanks{Supported by the bilateral project ``New perspectives on residuated posets'' of the Austrian Science Fund (FWF): project I 1923-N25, and the Czech Science Foundation (GA\v{C}R): project 15-34697L.
Partially supported by the projects ``Mathematical structures'' of  the Palack\'{y} University: projects IGA PrF 2014016 and IGA PrF 2015010.}


\begin{abstract}
Following 
[BOTUR, M.---CHAJDA, I.---HALA\v{S}, R.: \textit{Are basic algebras residuated structures?}, Soft Comput. \textbf{14} (2010), 251--255]
we discuss the connections between left-residuated partially ordered groupoids and the so-called basic algebras, which are a non-commutative and non-associative generalization of MV-algebras and orthomodular lattices.
\end{abstract}


\noindent
This is a preprint of an article published in Mathematica Slovaca. 
The final publication is available at \url{www.degruyter.com} and \url{https://doi.org/10.1515/ms-2016-0289}.

\newpage


\maketitle

This short note is in a sense a continuation of \cite{BoChHa} where M.~Botur, R.~Hala\v{s} and the first author studied the connections between bounded integral left-residuated lattice-ordered groupoids and the so-called basic algebras, which are a counterpart of bounded lattices with antitone involutions.
We try to recast and simplify these results and their proofs, and we present an alternative characterization of basic algebras in the setting of bounded integral left-residuated partially ordered groupoids.

Basic algebras were introduced by R.~Hala\v{s} and the authors in \cite{ChHaKu1} 
as an attempt at a common generalization of MV-algebras and orthomodular lattices and can be regarded as a non-commutative and non-associative extension of MV-algebras. 
Formally, a \emph{basic algebra}\footnote{We admit that the name ``basic algebras'' is unfortunate; it is meant only to indicate that these algebras are a common ``base'' for MV-algebras and orthomodular lattices, and it does not establish any link with H\'{a}jek's basic fuzzy logic and BL-algebras.} is an algebra $\al A=(A,\oplus,\neg,0)$ of type $(2,1,0)$ satisfying the equations\footnote{The original definition consisted of \eqref{ba1}--\eqref{ba4} and the equations $x\oplus \neg 0=\neg 0=\neg 0\oplus x$ which turned out to be redundant; see \cite{ChKo}.}
\begin{gather}
\label{ba1}
x\oplus 0=x,\\
\neg\neg x=x,\\
\label{ba3}
\neg(\neg x\oplus y)\oplus y=\neg(\neg y\oplus x)\oplus x,\\
\label{ba4}
\neg(\neg(\neg(x\oplus y)\oplus y)\oplus z)\oplus (x\oplus z)=\neg 0.
\end{gather}
Comparing these axioms with the axioms of MV-algebras (see, e.g., \cite{CDM}), it is not hard to see that MV-algebras are precisely the commutative and associative basic algebras.\footnote{In fact, MV-algebras are precisely the associative basic algebras since commutativity can be easily derived from \eqref{ba1}--\eqref{ba4} and associativity; see \cite{ChKu2} though this observation is due to M.~Kola\v{r}\'{i}k.}
Another non-associative generalization of MV-algebras that we studied earlier are NMV-algebras (see \cite{ChKu1}). 

If we put $1=\neg 0$ and define $x\leq y$ iff $\neg x\oplus y=1$, then $(A,\leq,0,1)$ is a bounded lattice with the property that, for every $a\in A$, $\gamma_a\colon x\mapsto \neg x\oplus a$ is an antitone involution on the principal filter $[a)$, and dually, $\delta_a\colon x\mapsto \neg (x\oplus\neg a)$ is an antitone involution on the principal ideal $(a]$. The lattice operations are given by $x\vee y=\neg (\neg x\oplus y)\oplus y$ and $x\wedge y=\neg (\neg x\vee\neg y)$.
Hence, basic algebras are closely related to the concept of a (\emph{bounded}) \emph{lattice with} (\emph{sectional}) \emph{antitone involutions} that appeared already in \cite{ChHaKu2005} and \cite{ChEm}. Originally, we considered bounded lattices with antitone involutions on principal filters, but it is quite obvious that when we are given antitone involutions on all principal filters, we are also given antitone involutions on all principal ideals, and vice versa.

Now, let $\al L=(L,\leq,0,1)$ be a bounded lattice with associated lattice operations $\vee$ and $\wedge$. If $\gamma_a$ ($a\in L$) are antitone involutions on the principal filters $[a)\subseteq L$, then the algebra $(L,\oplus,\neg,0)$, where 
\begin{equation}\label{baoperace1}
\neg x=\gamma_0(x) \quad\text{and}\quad x\oplus y=\gamma_y(\neg x\vee y),
\end{equation}
is a basic algebra whose underlying lattice is $\al L$ and where the induced antitone involutions on principal filters are just the $\gamma_a$'s.
Dually, if $\delta_a$ ($a\in L$) are antitone involutions on the principal ideals $(a]\subseteq L$, then the algebra $(L,\oplus,\neg,0)$ defined by
\begin{equation}\label{baoperace2}
\neg x=\delta_1(x) \quad\text{and}\quad x\oplus y=\neg \delta_{\neg y}(x\wedge\neg y)
\end{equation}
is a basic algebra whose underlying lattice is $\al L$ and where the induced antitone involutions on principal ideals are just the $\delta_a$'s.

In other words, any basic algebra is completely determined by its underlying lattice together with the antitone involutions. When we want to make a given structure into a basic algebra, it suffices to show that the structure is a bounded lattice with antitone involutions on all principal filters or ideals and then define the operations $\neg$ and $\oplus$ by \eqref{baoperace1} or \eqref{baoperace2}, respectively.

For instance, if $\al L=(L,\vee,\wedge,{}^\perp,0,1)$ is an orthomodular lattice, then for every $a\in L$, $\gamma_a\colon x\mapsto x^\perp\vee a$ is an antitone involution on $[a)$, hence by defining $\neg x=\gamma_0(x)=x^\perp$ and $x\oplus y=\gamma_y(\neg x\vee y)=(x^\perp\vee y)^\perp\vee y=(x\wedge y^\perp)\vee y$ we obtain a basic algebra. Note that $x\oplus y=x\vee y$ only if $\al L$ is a Boolean algebra.
Dually, $\delta_a\colon x\mapsto x^\perp\wedge a$ is an antitone involution on $(a]$, hence the corresponding basic algebra is defined by $\neg x=\delta_1(x)=x^\perp$ and $x\oplus y=\neg\delta_{\neg y}(x\wedge\neg y)=((x\wedge y^\perp)^\perp\wedge y^\perp)^\perp=(x\wedge y^\perp)\vee y$.
Orthomodular lattices are equivalent to basic algebras satisfying the quasi-equation $x\leq y$ $\Rightarrow$ $y\oplus x=y$, which is a direct translation of the orthomodular law $x\leq y$ $\Rightarrow$ $x\vee (x^\perp\wedge y)=y$.
However, this equivalence does not mean that a basic algebra fulfills the quasi-equation if and only if its underlying lattice equipped with $\neg$ is an orthomodular lattice; see \cite{ChKu2}.

For details and more information on basic algebras, besides the aforementioned papers, we refer the reader to \cite{ChKu3,KrKu,BoKu}; a short introduction can be also found in \cite{ChHaKu-book}.

The second basic concept is that of a bounded integral left-residuated partially ordered or lattice-ordered groupoid. A good reference to residuated structures is the monograph \cite{GJKO}.
A (\emph{bounded integral}) \emph{left-residuated po-groupoid} is a structure $\al G=(G,\leq,\cd,/,0,1)$ where $(G,\leq,0,1)$ is a bounded poset, $(G,\cd,1)$ is a groupoid with identity $1$, and $/$ is a binary operation on $G$ satisfying the left residuation law, for all $x,y,z\in G$:
\begin{equation}\label{lres}
x\cd y\leq z \quad\Leftrightarrow\quad x\leq z/y.
\end{equation}
If the poset is a lattice, we say that $\al G$ is a (\emph{bounded integral}) \emph{left-residuated \el groupoid}.\footnote{We are slightly abusing the terminology here. Usually, the definition of a po-groupoid requires that $\leq$ is preserved by multiplication on both sides, while the left residuation law \eqref{lres} only entails that $\leq$ is preserved by multiplication on the right. Similarly, the definition of an \el groupoid usually requires that multiplication distributes over $\vee$ from both sides, but \eqref{lres} only entails distributivity from the right. Thus we should more correctly speak of \emph{right} po-groupoids and \emph{right} \el groupoids.}
We do not assume anything else about the multiplication, it is neither commutative nor associative in general. Though it is usual to write $x\ra y$ instead of $y/x$ when the groupoid is commutative, we reserve $\ra$ for the implication in basic algebras; see below.

From now on, since all left-residuated po-groupoids (or \el groupoids) we consider in this paper are bounded and integral, we will omit the adjectives ``bounded integral''.

Now, given an arbitrary basic algebra $\al A=(A,\oplus,\neg,0)$, we define the following term operations---multiplication, implication and division:
\begin{equation}\label{baoperace3}
x\cd y=\neg (\neg x\oplus\neg y),\quad x\ra y=\neg x\oplus y \quad\text{and}\quad x/y=x\oplus\neg y.
\end{equation}
The lattice operations in the underlying lattice of $\al A$ are then expressed by
\begin{equation}\label{svazoperace}
x\vee y=(x\ra y)\ra y \quad\text{and}\quad x\wedge y=(x/y)\cd y.
\end{equation}
It is worth observing that in terms of the antitone involutions $\gamma_a\colon x\mapsto\neg x\oplus a$ on $[a)$ and $\delta_a\colon x\mapsto\neg (x\oplus\neg a)$ on $(a]$ we have:
$x\cd y=\neg \gamma_{\neg y}(x\vee\neg y) = \delta_y(\neg x\wedge y)$, 
$x\ra y=\gamma_y(x\vee y)=\neg\delta_{\neg y}(\neg x\wedge\neg y)$
and $x/y=\gamma_{\neg y}(\neg x\vee\neg y)=\neg \delta_y(x\wedge y)$.
It easily follows that $x\cd y\leq z$ iff $x\leq z/y$, for all $x,y,z\in A$. 
Therefore, the structure $\xal G A = (A,\leq,\cd,/,0,1)$ is a left-residuated \el groupoid, and the initial basic algebra $\al A$ can be retrieved from $\xal G A$ by $\neg x=0/x$ and $x\oplus y=\neg (\neg x\cd\neg y)=x/\neg y$.

This simple fact raises the question which left-residuated po-groupoids (or \el groupoids) correspond to basic algebras in this way, or in other words, under what conditions can a left-residuated po-groupoid be made into a basic algebra? This problem, especially in the commutative case, was tackled in \cite{BoChHa}. It was proved therein that if $\al G=(G,\leq,\cd,/,0,1)$ is a left\footnote{We should make a comment on terminology and notation. What we now call ``left-residuated'' was called ``right-residuated'' in \cite{BoChHa} and our $x/y$ was denoted by $y\ra x$, thus the residuation law of \cite{BoChHa} was as follows: $x\cd y\leq z$ iff $x\leq y\ra z$.}-residuated \el groupoid satisfying 
\begin{enumerate}[\indent (a)]
\item
the divisibility law $(x/y)\cd y=x\wedge y$, 
\item
the double negation law $\neg\neg x=x$ with negation defined by $\neg x=0/x$, and 
\item
the condition called CAP, for all $x,y,z\in G$:
\begin{equation}\label{cap}
\neg x/y\leq \neg z \quad\Leftrightarrow\quad z\leq x\cd y,
\end{equation}
\end{enumerate}
then the algebra $\xal A G = (G,\oplus,\neg,0)$, where $x\oplus y=\neg (\neg x\cd\neg y)$, is a basic algebra. Since the left-residuated \el groupoid $\xxal G A G$ associated with $\xal A G$ is just $\al G$, basic algebras are equivalent to left-residuated \el groupoids satisfying (a), (b) and (c).\footnote{Actually, it is implicit in \cite{BoChHa} that for any basic algebra $\al A$, $\xal G A$ is a left-residuated \el groupoid such that $\xxal A G A = \al A$; the proof is given for commutative basic algebras only.}  

Our first goal is to show that we do not need to assume that $\al G$ is lattice-ordered (cf. Lemma~\ref{lemma1} below) and that the somewhat opaque condition \eqref{cap} may be replaced with a single simple equation which also captures the law of double negation (cf. Lemma~\ref{lemma2}).

When doing calculations, we will use the following rules that are simple consequences of \eqref{lres}:

\begin{lemma}
In any left-residuated po-groupoid $\al G$, for all $x,y,z\in G$:
\begin{enumerate}[\indent\upshape (a)]
\item $(x/y)\cdot y\leq x\leq (x\cdot y)/y$;
\item $x\leq y$ iff $y/x=1$;
\item $x\leq y$ implies $x\cd z\leq y\cd z$;
\item $x\cdot y\leq y$;
\item $x/x=1/x=1$ and $x/1=x$;
\item $x\leq y$ implies $x/z\leq y/z$;
\item $x\cdot 0=0=0\cdot x$;
\item $\big(\bigvee_{i\in I} x_i\big)\cd y=\bigvee_{i\in I} (x_i\cd y)$ if $\bigvee_{i\in I} x_i$ exists;
\item $\big(\bigwedge_{i\in I} x_i\big)/y=\bigwedge_{i\in I} (x_i/y)$ if $\bigwedge_{i\in I} x_i$ exists.
\end{enumerate}
\end{lemma}

\begin{lemma}\label{lemma1}
Let $\al G$ be a left-residuated po-groupoid. The following are equivalent:
\begin{enumerate}[\indent\upshape (a)]
\item $\al G$ satisfies the divisibility law
\begin{equation}\label{div}
(x/y)\cdot y=(y/x)\cdot x\text{;}
\end{equation}
\item for any $x,y\in G$, $x\leq y$ iff $x=z\cdot y$ for some $z\in G$;
\item the underlying poset of $\al G$ is a meet-semilattice in which $x\wedge y=(x/y)\cdot y$.
\end{enumerate}
\end{lemma}

\begin{proof}
(a) $\Rightarrow$ (b). If $x\leq y$, then $y/x=1$ and $x=1\cdot x=(y/x)\cdot x = (x/y)\cdot y$. Conversely, if $x=z\cdot y$ for some $z$, then $x\leq y$.

(b) $\Rightarrow$ (c). Clearly, $(x/y)\cdot y$ is a lower bound of $\{x,y\}$. If $z$ is another lower bound of $\{x,y\}$, there exists $u\in G$ such that $z=u\cdot y\leq x$. Then $u\leq x/y$ and $z=u\cdot y \leq (x/y)\cdot y$.
Thus $(x/y)\cdot y=\inf\{x,y\}$.

(c) $\Rightarrow$ (a). This is obvious.
\end{proof}

\begin{lemma}\label{lemma2}
A left-residuated po-groupoid $\al G$ satisfies the law of double negation and the condition \eqref{cap} if and only if 
it satisfies the equation 
\begin{equation}\label{jk}
x\cd y=\neg (\neg x/y).
\end{equation}
In this case, for all $x,y\in G$, we have $x\leq y$ iff $\neg y\leq\neg x$.
\end{lemma}

\begin{proof}
Suppose that $\al G$ satisfies $\neg\neg x=x$ and \eqref{cap}.
Letting $y=1$ in \eqref{cap} we get $\neg x\leq\neg z$ iff $z\leq x$ for all $x,z\in G$. Hence \eqref{cap} can be rewritten as follows: for any fixed $x,y\in G$, $z\leq\neg (\neg x/y)$ iff $z\leq x\cd y$ for all $z\in G$. This means $\neg (\neg x/y)=x\cd y$.

Conversely, suppose that $\al G$ satisfies \eqref{jk}. For every $x\in G$ we have $x=x\cd 1=\neg(\neg x/1)=\neg\neg x$.
Hence if $x\leq y$, then $0=\neg 1=\neg (y/x)=\neg y\cdot x$ by \eqref{jk}. 
But $\neg y\cdot x\leq 0$ iff $\neg y\leq 0/x=\neg x$, and so $x\leq y$ implies (in fact, is equivalent to) $\neg y\leq\neg x$.
Now, for any $x,y,z\in G$ we have $z\leq x\cd y=\neg(\neg x/y)$ iff $\neg z\geq\neg x/y$, which is \eqref{cap}.
\end{proof}

\begin{theorem}\label{thm1}
The mutually inverse assignments $\al A\mapsto \xal G A$ and $\al G\mapsto \xal A G$ establish an equivalence between basic algebras and (bounded integral) left-resi\-du\-ated po-groupoids satisfying the equations \eqref{div} and \eqref{jk}.
\end{theorem}

\begin{proof}
Using Lemmas \ref{lemma1} and \ref{lemma2}, this is just a reformulation of the quoted results of \cite{BoChHa}. 
The proofs given in \cite{BoChHa} consist in tedious verifying that the algebra $\xal A G$ associated with a left-residuated po-groupoid $\al G$ satisfies the axioms \eqref{ba1}--\eqref{ba4}, but if we think of basic algebras as lattices with antitone involutions on principal ideals, the proofs become almost trivial.

Let $\al G$ satisfy \eqref{div} and \eqref{jk}. For any $a\in G$, the map $\delta_a\colon x\mapsto \neg x\cd a=\neg (x/a)$ is an antitone involution on the principal ideal $(a]$. Indeed, $\neg x\cd a\leq a$ for any $x$, and for $x\leq y\leq a$ we have $\delta_a(x)=\neg x\cd a\geq\neg y\cd a=\delta_a(y)$, and also $\delta_a(\delta_a(x))=\neg(\neg x\cd a)\cd a=(x/a)\cd a=x\wedge a=x$.
Moreover, since the underlying poset of $\al G$ is a meet-semilattice by Lemma \ref{lemma1}, it follows that it is actually a lattice in which $x\vee y=\neg(\neg x\wedge\neg y)$.
Therefore, $\al G$ can be made into a basic algebra by means of \eqref{baoperace2}, i.e., the negation $\neg x=\delta_1(x)$ remains unchanged and the addition is 
$x\oplus y=\neg \delta_{\neg y}(x\wedge\neg y)=(x\wedge\neg y)/\neg y=x/\neg y=\neg (\neg x\cd\neg y)$.
But this basic algebra is precisely the $\xal A G$. 
It is easily seen that $\xxal G A G = \al G$.

On the other hand, it follows directly from \eqref{baoperace3} and \eqref{svazoperace} that for any basic algebra $\al A$, the left-residuated po-groupoid $\xal G A$ satisfies both \eqref{div} and \eqref{jk}, and $\xxal A G A = \al A$.
\end{proof}

As a particular case of Theorem \ref{thm1} we obtain the known equivalence between MV-algebras and residuated commutative po-monoids satisfying divisibility and the law of double negation. It suffices to note that $\al A$ is commutative and associative if and only if $\xal G A$ is commutative and associative, in which case $\neg x/y=(0/x)/y=0/(x\cd y)=\neg(x\cd y)$, and so \eqref{jk} is equivalent to the equation $x\cd y=\neg\neg (x\cd y)$, which is in turn the same as the law of double negation.

However, in our Theorem \ref{thm1}, the equation \eqref{jk} cannot be replaced by $\neg\neg x=x$ because there exist left-residuated po-groupoids $\al G$ satisfying both divisibility and the law of double negation, whereas the algebra $\xal A G$ is not a basic algebra.

\begin{example}
Let $\al G$ be the following left-residuated po-groupoid (we found it with help of Prover9-Mace4):

\medskip

\begin{center}
\begin{minipage}{5.5cm}
\begin{tabular}{c|cccccccc}
$\cdot$ &	$0$ & $a$ & $b$ & $c$ & $d$ & $e$ & $f$ & $1$\\ \hline
$0$ & 		$0$ & $0$ & $0$ & $0$ & $0$ & $0$ & $0$ & $0$\\
$a$ & 		$0$ & $0$ & $0$ & $0$ & $b$ & $c$ & $0$ & $a$\\
$b$ & 		$0$ & $0$ & $b$ & $c$ & $0$ & $0$ & $a$ & $b$\\
$c$ & 		$0$ & $a$ & $0$ & $0$ & $d$ & $e$ & $b$ & $c$\\
$d$ & 		$0$ & $0$ & $b$ & $c$ & $0$ & $a$ & $a$ & $d$\\
$e$ & 		$0$ & $a$ & $0$ & $a$ & $d$ & $e$ & $d$ & $e$\\
$f$ & 		$0$ & $0$ & $b$ & $c$ & $b$ & $c$ & $a$ & $f$\\
$1$ & 		$0$ & $a$ & $b$ & $c$ & $d$ & $e$ & $f$ & $1$
\end{tabular}
\end{minipage}
\quad
\begin{minipage}{3cm}
\begin{tikzpicture}[scale=0.8]
\draw (0,1.5) -- (1,0.75) -- (1,-0.75) -- (0,-1.5) -- (-1,-0.75) -- (-1,0.75) -- (0,1.5);
\draw (1,0.75) -- (-1,-0.75);
\fill (0,1.5) circle (2pt) node[above right] {$1$};
\fill (0,-1.5) circle (2pt) node[below right] {$0$};
\fill (-1,0.75) circle (2pt) node[left] {$e$};
\fill (-1,0) circle (2pt) node[left] {$c$};
\fill (1,0) circle (2pt) node[right] {$d$};
\fill (1,0.75) circle (2pt) node[right] {$f$};
\fill (-1,-0.75) circle (2pt) node[left] {$a$};
\fill (1,-0.75) circle (2pt) node[right] {$b$};
\end{tikzpicture}
\end{minipage}
\end{center}

\medskip

\noindent
We have $\neg 0=1$, $\neg a=f$, $\neg b=e$, $\neg c=c$, $\neg d=d$, $\neg e=b$, $\neg f=a$ and $\neg 1=0$, hence $\al G$ satisfies the double negation law.
It is easily seen that $\al G$ satisfies the condition (b) of Lemma \ref{lemma1}, which is equivalent to the divisibility law \eqref{div}. 
However, $\al G$ does not satisfy the equation \eqref{jk}; for instance, $c\cdot b=0$, but $c/b=e$ and so $\neg (\neg c/b)=\neg e=b$.
The algebra $\xal A G$, in which $x\oplus y=\neg (\neg x\cdot\neg y)$, is not a basic algebra. For instance, $\neg (\neg c\oplus e)\oplus e=e$ but $\neg (\neg e\oplus c)\oplus c=c$, thus $\xal A G$ does not satisfy \eqref{ba3}.
\end{example}

In what follows, we would like to present an axiomatization of left-residuated po-groupoids corresponding to basic algebras that would be based on the implication $\ra$ defined by \eqref{baoperace3}, 
i.e., $x\ra y=\neg x\oplus y$. To this end, for any left-residuated po-groupoid $\al G=(G,\leq,\cd,/,0,1)$, in addition to the already defined negation $\neg x=0/x$, we define implication by
\begin{equation}\label{implikace}
x\ra y=\neg x/\neg y.
\end{equation}
It is obvious that $x\ra 0=\neg x$ and $x\ra x=x\ra 1=1$, while $1\ra x=\neg\neg x$.

Recalling \eqref{svazoperace}, we see that all left-residuated po-groupoids corresponding to basic algebras must satisfy the equation
\begin{equation}\label{w}
(x\ra y)\ra y=(y\ra x)\ra x,
\end{equation}
and that the antitone involutions on principal filters $[a)$ must be given by $\gamma_a\colon x\mapsto x\ra a$. Now, we prove that the law of double negation and \eqref{w} are also sufficient for $\xal A G$ to be a basic algebra.

\begin{lemma}\label{lemma3}
Let $\al G$ be a left-residuated po-groupoid satisfying the double negation law and the equation \eqref{w}.
Then, for all $x,y,z\in G$, we have:
\begin{enumerate}[\indent\upshape (a)]
\item $1\ra x=x$;
\item $x\leq y\ra x$;
\item $x\leq y$ iff $\neg y\leq\neg x$ iff $x\ra y=1$;
\item if $x\leq y$, then $y\ra z\leq x\ra z$.
\end{enumerate}
\end{lemma}

\begin{proof}
The proofs are but direct calculations using \eqref{implikace} and the double negation law:
\begin{enumerate}[\indent (a)]
\item $1\ra x=\neg\neg x=x$.
\item $x\leq\neg\neg x=0/\neg x$ implies $x\cd\neg x=0\leq\neg y$, whence $x\leq\neg y/\neg x=y\ra x$.
\item If $x\leq y$, then $1=y/x=\neg y\ra\neg x$, whence $\neg x=1\ra\neg x=(\neg y\ra\neg x)\ra\neg x=(\neg x\ra\neg y)\ra\neg y\geq \neg y$ by (a) and (b). Thus, $x\leq y$ iff $\neg y\leq\neg x$ iff $1=\neg x/\neg y=x\ra y$.
\item By (c), $x\leq y$ iff $\neg y\leq\neg x$, which yields $y\ra z=\neg y/\neg z\leq\neg x/\neg z=x\ra z$.
\end{enumerate}
\end{proof}

\begin{lemma}\label{lemma4}
Let $\al G$ be a left-residuated po-groupoid satisfying the double negation law and the equation \eqref{w}. Then the underlying poset of $\al G$ is a join-semilattice in which $x\vee y=(x\ra y)\ra y$. Moreover, for every $a\in G$, the map $\gamma_a\colon x\mapsto x\ra a$ is an antitone involution on the principal filter $[a)$.
\end{lemma}

\begin{proof}
By Lemma \ref{lemma3} (b), the element $(x\ra y)\ra y=(y\ra x)\ra x$ is an upper bound of $\{x,y\}$.
Suppose that $x,y\leq z$.
Then, by Lemma \ref{lemma3} (d), $x\leq z$ yields $x\ra y\geq z\ra y$ whence $(x\ra y)\ra y\leq (z\ra y)\ra y=(y\ra z)\ra z=1\ra z=z$ since $y\leq z$ iff $y\ra z=1$ by (c).
Thus $(x\ra y)\ra y=(y\ra x)\ra x=\sup\{x,y\}$.

For the latter statement, we fix $a\in G$. Since $x\ra a\geq a$, $\gamma_a$ is well-defined.
If $a\leq x\leq y$, then $\gamma_a(x)=x\ra a\geq y\ra a=\gamma_a(y)$, and $\gamma_a(\gamma_a(x))=(x\ra a)\ra a=x\vee a=x$. Hence $\gamma_a$ is an antitone involution on $[a)$.
\end{proof}

\begin{theorem}\label{thm2}
The assignments $\al A\mapsto \xal G A$ and $\al G\mapsto \xal A G$ establish an equivalence between basic algebras and (bounded integral) left-residuated po-groupoids satisfying the law of double negation and the equation \eqref{w}.
\end{theorem}

\begin{proof}
We know that if $\al A$ is a basic algebra, then $\xal G A$ satisfies the law of double negation as well as \eqref{w}. Clearly, $\xxal A G A = \al A$.

Conversely, let $\al G=(G,\leq,\cd,/,0,1)$ be a left-residuated po-groupoid satisfying the two equations. In view of Lemma \ref{lemma4}, the underlying poset is a lattice where $x\vee y=(x\ra y)\ra y$ and $x\wedge y=\neg(\neg x\vee\neg y)$, because it is a join-semilattice and $\gamma_0\colon x\mapsto\neg x$ is an antitone involution.
Now, $\al G$ can be made into a basic algebra, say $\al B=(G,\oplus,\neg,0)$, by \eqref{baoperace1}: $x\oplus y=\gamma_y(\neg x\vee y)=(\neg x\vee y)\ra y=((\neg x\ra y)\ra y)\ra y=(\neg x\ra y)\vee y=\neg x\ra y=x/\neg y$, and the negation in $\al B$ coincides with the one in $\al G$.
By \eqref{baoperace3}, in $\xal G B=(G,\preceq,*,{\sslash},0,1)$ we have:
$x {\sslash} y=x\oplus\neg y=x/y$, so $x\preceq y$ iff $x\leq y$, and likewise $x*y=\neg(\neg x\oplus\neg y)=x\cd y$, because for any $x,y,z\in G$, 
$x*y\leq z$ iff $x\leq z {\sslash} y=z/y$ iff $x\cd y\leq z$.
Consequently, $\xal G B = \al G$ and so $\al B = \xxal A G B = \xal A G$.
Finally, it is obvious that $\xxal G A G =\al G$.
\end{proof}
 
The implication $\ra$ plays a central role in basic algebras, but it is apparent that if we want to stick to  $\ra$ instead of $/$, we have to modify the left residuation law somehow, e.g., as follows:
By a (\emph{bounded integral}) \emph{contrapositionally residuated po-groupoid} we mean a structure $\al G=(G,\leq,\cdot,\ra,0,1)$ where $(G,\leq,0,1)$ is a bounded poset, $(G,\cdot,1)$ is a groupoid with identity, and $\ra$ is a binary operation on $G$ such that, for all $x,y,z\in G$:
\begin{enumerate}[\indent (a)]
\item $1\ra x=x$,
\item $x\cdot y\leq z$ iff $x\leq {\sim} z\ra {\sim} y$, where we write ${\sim} x$ for $x\ra 0$.
\end{enumerate}
Obviously, when we put $x/y={\sim} x \ra {\sim} y$, then $\al G' = (G,\leq,\cdot,/,0,1)$ is a left-residuated po-groupoid. 
Also, for any basic algebra $\al A=(A,\oplus,\neg,0)$, the structure $(A,\leq,\cd,\ra,0,1)$ is a contrapositionally residuated po-groupoid where ${\sim} x=\neg x$.

The equation (a), which does not follow from the ``contrapositional'' residuation law (b), guarantees that the negation ${\sim}x=x\ra 0$ in $\al G$ is the same as the negation $\neg x=0/x$ in $\al G'$, and also that $\al G$ (and $\al G'$) satisfies the law of double negation.
Indeed, ${\sim} 1=1\ra 0=0$, whence ${\sim\sim} x={\sim} x\ra 0={\sim} x\ra {\sim} 1=x/1=x$, thus ${\sim} 0=1$ and $\neg x=0/x={\sim} 0\ra {\sim} x=1\ra {\sim} x={\sim} x$. 
In fact, it is possible to replace (a) with the equation ${\sim} x=\neg x$ in the definition.

\begin{proposition}
Contrapositionally residuated po-groupoids are equivalent to left-resi\-du\-ated po-groupoids satisfying the double negation law $\neg\neg x=x$.
\end{proposition}

\begin{proof}
Let $\al G=(G,\leq,\cdot,\ra,0,1)$ be a contrapositionally residuated po-groupoid and let $x/y={\sim} x\ra {\sim} y$.
By the above discussion, $\al G' = (G,\leq,\cd,/,0,1)$ is a left-residuated po-groupoid in which the law of double negation holds. We have $\neg x/\neg y={\sim}\neg x\ra {\sim}\neg y=x\ra y$.

Conversely, let $\al G=(G,\leq,\cdot,/,0,1)$ be a left-residuated po-groupoid that satisfies the law of double negation and let $x\ra y=\neg x/\neg y$.
Then $1\ra x=\neg 1/\neg x=0/\neg x=\neg\neg x=x$,
${\sim} x=x\ra 0=\neg x/\neg 0=\neg x/1=\neg x$ and 
${\sim} x\ra {\sim} y=\neg{\sim} x/\neg{\sim} y=x/y$. 
Hence $z\cdot y\leq x$ iff $z\leq x/y={\sim} x\ra {\sim} y$, which means that $\al G'=(G,\leq,\cdot,\ra,0,1)$ is a contrapositionally residuated po-groupoid.
\end{proof}

We end with two more remarks.

(a) In \cite{BoChHa}, special attention was devoted to commutative basic algebras. Commutative residuated \el groupoids corresponding to commutative basic algebras were characterized by means of ``skew divisibility'' $(\neg y/\neg x)\cd y = x\wedge y$ and the double negation law $\neg\neg x=x$.
Recently, a slightly different axiomatization was presented in \cite{ChHa}: $(x/y)\cd y=x\wedge y$, $x/y=\neg y/\neg x$ and $\neg\neg x=x$. In both cases, $\neg\neg x=x$ is redundant because it follows from $(\neg y/\neg x)\cd y = x\wedge y$ or $x/y=\neg y/\neg x$, respectively, by substituting $y=1$.
Moreover, it is easy to see that the skew divisibility implies the contraposition law $x/y=\neg y/\neg x$ and the standard divisibility (because $\neg y/\neg x\leq x/y$ iff $x\wedge y=(\neg y/\neg x)\cd y\leq x$), while the converse implication is trivial.

Also, when our Theorem \ref{thm1} is restricted to the commutative case, the equation \eqref{jk} can be equivalently replaced with $x/y=\neg y/\neg x$. Indeed, if a commutative residuated po-groupoid $\al G$ satisfies \eqref{jk}, then $x/y=\neg (\neg x\cd y)=\neg (y\cd\neg x)=\neg y/\neg x$. 
On the other hand, if $\al G$ satisfies the contraposition law, then $x\leq (x\cd y)/y=\neg y/\neg (x\cd y)$ implies $\neg (x\cd y)\leq\neg y/x$, 
and at the same time, $x\leq \neg y/(\neg y/x)=\neg(\neg y/x)/y$ implies $x\cd y\leq\neg (\neg y/x)$, thus $x\cd y=\neg (\neg y/x)$.

(b) It was noticed in \cite{BoChHa} that basic algebras are not \emph{right} residuated in general. 
More precisely, given a basic algebra $\al A = (A,\oplus,\neg,0)$ with the order and multiplication defined as before, there may be no binary operation $\bs$ on $A$ such that $x\cd y\leq z$ iff $y\leq x\bs z$ for all $x,y,z\in A$. The point is that the existence of such a right residuum (left division) would imply that $\leq$ is preserved by multiplication on the left, i.e., $x\leq y$ $\Rightarrow$ $z\cd x\leq z\cd y$, or equivalently, $x\leq y$ $\Rightarrow$ $z\oplus x\leq z\oplus y$. This condition, however, holds only in the so-called \emph{monotone} basic algebras; see \cite{BoKu,ChKu3,KrKu,BoKuRa}. 

The following example shows that monotonicity is a necessary condition only. The example given in \cite{BoChHa} is finite and since by \cite{BoKu}, Theorem 4.7, every finite monotone\footnote{A stronger result was proved in \cite{BoKu}: every finite basic algebra satisfying $x\leq x\oplus y$ is an MV-algebra.} basic algebra is an MV-algebra, we need to find an infinite monotone basic algebra that is not right-residuated.

\begin{example}
Let $A=[0,1]\subseteq\mathbb{R}$ be equipped with antitone involutions $\delta_a$ on $[0,a]$ as follows:
$\delta_1(x)=\sqrt{1-x^2}$, and $\delta_a(x)=a-x$ for $a<1$. 
Let $\al A$ be the basic algebra defined by \eqref{baoperace2}, i.e., $\neg x=\delta_1(x)$ and 
$x\cdot y=\neg (\neg x\oplus\neg y)=\delta_y(\neg x\wedge y)$. 
Note that $x\cd 1=x$ and $x\cd y=y-(\neg x\wedge y)=(y-\neg x)\vee 0$ for $y<1$.
 
The algebra $\al A$ is monotone. Indeed, $x<y<1$ yields $z\cdot x=(x-\neg z)\vee 0\leq (y-\neg z)\vee 0=z\cdot y$, and also $z\cdot x=(x-\neg z)\vee 0\leq z=z\cdot 1$ as $x-z\leq 1-z\leq\neg z$.

Let $x\in A\setminus\{0,1\}$ be fixed, put $y=1-\neg x$ and suppose that $x\backslash y$ exists, i.e., $z\leq x\backslash y$ iff $x\cdot z\leq y$ for all $z\in A$.
If $z=1$, then $x\cdot 1=x\nleq y$ as $x\leq y$ iff $\neg x\leq 1-x$ iff $x\in\{0,1\}$.
If $z<1$, then $x\cdot z=(z-\neg x)\vee 0\leq 1-\neg x=y$. 
Therefore, the right residuation law would yield $x\backslash y\neq 1$ and $z\leq x\backslash y$ for all $z<1$, which is a contradiction.
\end{example}


\end{document}